\documentclass[reqno,11pt]{amsart}
\pdfoutput1

\usepackage{amssymb,color,hyperref,mathrsfs,stmaryrd}
\usepackage{amsmath}

\usepackage[usenames,dvipsnames]{xcolor}

\usepackage[curve,matrix,arrow]{xy}

\numberwithin{equation}{section}
\newtheorem{theorem}{Theorem}[section]
\newtheorem{lemma}[theorem]{Lemma}

\newtheorem{Th}{Theorem}
\newtheorem{corollary}[Th]{Corollary}

\theoremstyle{definition}

\newtheorem{definition}[theorem]{Definition}

\newtheorem{remark}[equation]{Remark}

\newcommand{\F}{\mathcal{F}}
\newcommand{\E}{\mathcal{E}}
\newcommand{\G}{\mathcal{G}}
\renewcommand{\L}{\mathcal{L}}

\newcommand{\tGamma}{\tilde{\Gamma}}

\newcommand{\tK}{\tilde{\mathcal{K}}}
\newcommand{\tn}{\tilde{n}}
\newcommand{\M}{\mathcal{M}}
\newcommand{\N}{\mathcal{N}}
\newcommand{\K}{\mathcal{K}}
\newcommand{\D}{\mathbf{D}}
\renewcommand{\H}{\mathcal{H}}

\newcommand{\Aut}{\operatorname{Aut}}

\newcommand{\ov}{\overline}

\newcommand{\Ac}{\operatorname{A}^\circ}
\newcommand{\One}{\operatorname{\mathbf{1}}}
\newcommand{\W}{\mathbf{W}}
\newcommand{\fN}{\mathfrak{N}}
\newcommand{\X}{{\mathcal{X}}}
\newcommand{\Y}{{\mathcal{Y}}}
\newcommand{\mD}{\mathcal{D}}
\newcommand{\tE}{\widetilde{\E}}
\newcommand{\fS}{\mathfrak{S}}

\def \<{\langle }
\def \>{\rangle }

\renewcommand{\phi}{\varphi}

\newcommand{\subn}{{\unlhd\!\unlhd\;}}

\title[Some products in fusion systems and localities]{Some products in fusion systems and localities}

\author[E.~Henke]{Ellen Henke}
\address{Institut f{\"u}r Algebra, Fakult{\"a}t Mathematik, Technische Universit{\"a}t Dresden, 01062 Dresden, Germany}
\email{ellen.henke@tu-dresden.de}

\begin{document}

\begin{abstract}
The theory of saturated fusion systems resembles in many parts the theory of finite groups. However, some concepts from finite group theory are difficult to translate to fusion systems. For example, products of normal subsystems with other subsystems are only defined in special cases. In this paper the theory of localities is used to prove the following result: Suppose $\F$ is a saturated fusion system over a $p$-group $S$. If $\E$ is a normal subsystem of $\F$ over $T\leq S$, and $\mD$ is a subnormal subsystem of $N_\F(T)$ over $R\leq S$, then there is a subnormal subsystem $\E\mD$ of $\F$ over $TR$, which plays the role of a product of $\E$ and $\mD$ in $\F$. If $\mD$ is normal in $N_\F(T)$, then $\E\mD$ is normal in $\F$. It is shown along the way that the subsystem $\E\mD$ is closely related to a naturally arising product in certain localities attached to $\F$.
\end{abstract}

\maketitle

\section{Introduction}

The theory of localities was developed by Chermak \cite{Chermak:2013, Chermak:2015} and gives an alternative way of looking at linking systems and transporter systems. Here linking systems and transporter systems are categories associated to saturated fusion systems that were defined and studied before by various authors; see \cite{BLO2,BCGLO1,O4,OV1}. Linking systems are particularly important in the study of the homotopy theory of fusion systems. For an introduction to localities the reader is referred to \cite{Chermak:2015} or to the summary in \cite[Sections~3.1-3.5 and 3.7]{Henke:Regular}.

\smallskip

Many concepts and results from finite group theory can be translated both to saturated fusion systems and to a certain class of localities; for saturated fusion systems this is largely due to Aschbacher \cite{Aschbacher:2008, Aschbacher:2011}. It is shown in \cite{Chermak/Henke} that the two theories are closely related. This can be used to reprove known results on fusion systems and, perhaps more importantly, to show some properties of fusion systems that seemed difficult to prove before. The present paper gives a further example of this. 

\smallskip

To explain our approach and some results from \cite{Chermak/Henke} in more detail let us first introduce some definitions. If $(\L,\Delta,S)$ is a locality and $\H$ is a partial subgroup of $\L$, then we can always naturally form the fusion system $\F_{S\cap\H}(\H)$. By definition, this is the fusion system over $S\cap\H$, which is generated by the maps between subgroups of $S\cap\H$ that are conjugation maps by elements of $\H$. The locality $(\L,\Delta,S)$ is said to be a locality \emph{over} $\F$ if $\F=\F_S(\L)$. 

\smallskip

As in \cite{Henke:2015} we call a locality $(\L,\Delta,S)$  a \emph{linking locality} if $\F_S(\L)$ is saturated, $\F_S(\L)^{cr}\subseteq\Delta$ and $N_\L(P)$ is a group of characteristic $p$ for every $P\in\Delta$. Here a finite group $G$ is of characteristic $p$ if $C_G(O_p(G))\leq O_p(G)$. We caution the reader that Chermak \cite{ChermakII,ChermakIII} calls such localities \emph{proper localities} and this terminology is also used in \cite{Chermak/Henke}. If $\F$ is a saturated fusion system over a $p$-group $S$, it is a consequence of the existence and uniqueness of centric linking systems that a linking locality $(\L,\Delta,S)$ over $\F$ always exists (cf. \cite{Chermak:2013,Oliver:2013,Glauberman/Lynd}). Indeed, for any ``suitable'' set $\Delta$  of subgroups of $S$, it is shown in \cite[Theorem~A]{Henke:2015} that there is a linking locality $(\L,\Delta,S)$ over $\F$ with object set $\Delta$. Examples for suitable object sets are the sets $\F^c$, $\F^q$ and $\F^s$ of $\F$-centric, $\F$-quasicentric and $\F$-subcentric subgroups of $S$ respectively; see \cite[Definition~I.3.1, Definition~III.4.5]{Aschbacher/Kessar/Oliver:2011} and \cite[Definition~1]{Henke:2015}. The set of $\F$-subcentric subgroups turns out to be the largest possible object set of a linking locality. Another particularly nice object set is the set $\delta(\F)$ of regular objects, which was first defined by Chermak \cite[p.36]{ChermakIII}. The reader is referred to Subsection~\ref{SS:Regular} for details. A linking locality $(\L,\Delta,S)$ over $\F$ is called \emph{regular} if $\Delta=\delta(\F)$.

\smallskip

Suppose now that $\F$ is a saturated fusion system and $(\L,\Delta,S)$ is a linking locality over $\F$. Write $\fN(\F)$ for the set of normal subsystems of $\F$ and $\fN(\L)$ for the set of partial normal subgroups of $\L$. Chermak and the author of this paper \cite[Theorem~A]{Chermak/Henke} proved that there is a bijection
\[\Psi_\L\colon\fN(\L)\rightarrow \fN(\F)\]
which sends every partial normal subgroup $\N\in\fN(\L)$ to a normal subsystem over $S\cap \N$, which is the smallest normal subsystem of $\F$ containing $\F_{S\cap\N}(\N)$. If $\F^q\subseteq\Delta$ or $\delta(\F)\subseteq \Delta$, then it even turns out that $\Psi_\L$ is given by $\Psi_\L(\N)=\F_{S\cap\N}(\N)$ for all $\N\in\fN(\L)$.

\smallskip

In fact, if $(\L,\Delta,S)$ is a \emph{regular locality} over $\F$ then, writing $\fS(\F)$ for the set of subnormal subsystems of $\F$ and $\fS(\L)$ for the set of partial subnormal subgroups of $\L$, the map
\[\hat{\Psi}_\L\colon \fS(\L)\rightarrow \fS(\F),\H\mapsto \F_{S\cap\H}(\H)\]
is well-defined and a bijection by \cite[Theorem~F]{Chermak/Henke}. Note that $\hat{\Psi}_\L$ restricts to $\Psi_\L$.

\smallskip

The correspondence between normal subsystems and partial normal subgroups is used in \cite[Theorem~C]{Chermak/Henke} to show that, for any two normal subsystems $\E_1,\E_2$ of $\F$, there is a meaningful notion of a product subsystem $\E_1\E_2$. The proof uses that products in localities can be formed very naturally. Namely, if $\X,\Y\subseteq\L$ and $\Pi\colon\D\rightarrow\L$ denotes the partial product on $\L$, we set
\[\X\Y:=\Pi(\X,\Y):=\{\Pi(x,y)\colon x\in\X,\;y\in\Y,\;(x,y)\in\D\}.\]
It turns out that the product of any two partial normal subgroups of a locality is again a partial normal subgroup; see \cite{Henke:2015a}. Because of the existence of the bijection $\Psi_\L$ this makes it possible to define products of normal subsystems of fusion systems.

\smallskip

In the present paper we follow a relatively similar strategy to the one just described, but in a slightly different context. We are interested in the situation that $\E$ is a normal subsystem of $\F$ over a subgroup $T\leq S$ and that $\mD$ is a subnormal subsystem of $N_\F(T)$. If $(\L,\Delta,S)$ is a linking locality over $\F$ as before, then by \cite[Lemma~3.35(b)]{Henke:Regular}, $(N_\L(T),\Delta,S)$ is a linking locality over $N_\F(T)$. Therefore, there is also a bijection
\[\Psi_{N_\L(T)}\colon \fN(N_\L(T))\rightarrow \fN(N_\F(T))\]
with similar properties as mentioned before for $\Psi_\L$. Indeed, assuming that $(\L,\Delta,S)$ is regular, we show in Lemma~\ref{L:NLTRegular} that $(N_\L(T),\delta(N_\F(T)),S)$ is a regular locality, and so the map
\[\hat{\Psi}_{N_\L(T)}\colon \fS(N_\L(T))\rightarrow \fS(N_\F(T)),\H\mapsto \F_{S\cap\H}(\H)\]
is well-defined and a bijection. Thus, in the situation described above, $\E$ corresponds to a partial normal subgroup $\N$ of $\L$ via $\Psi_\L$, and $\mD$ corresponds to a partial subnormal subgroup $\K$ of $N_\L(T)$ via $\hat{\Psi}_{N_\L(T)}$. We use then Theorem~\ref{C:NK} below to conclude that $\N\K$ is a partial subnormal subgroup of $\L$ and corresponds thus to a subnormal subsystem of $\F$ via $\hat{\Psi}_\L$. The proof of Theorem~\ref{T:RegularSubnormal} builds on the following theorem, which holds for arbitrary localities.

\begin{Th}\label{P:Localities}
Let $(\L,\Delta,S)$ be a locality, $\N\unlhd\L$, $T:=S\cap\N$ and $\K\unlhd N_\L(T)$. Then $\N\K$ is a partial normal subgroup of $\L$, $\N\K\cap S=T(\K\cap S)$ and $\N\K=\K\N$. Moreover, for every $g\in\N\K$ the following hold:
\begin{itemize}
\item [(a)] There exist $n\in\N$ and $k\in\K$ with $(n,k)\in\D$, $g=nk$ and $S_g=S_{(n,k)}$.
\item [(b)] There exist $n\in\N$ and $k\in\K$ with $(k,n)\in\D$, $g=kn$ and $S_g=S_{(k,n)}$.
\end{itemize}
\end{Th}

While the theorem above appears to be new, a somewhat similar (but slightly weaker) result was shown before by Chermak. Namely, if $(\L,\Delta,S)$ is a linking locality, Proposition~5.5 in \cite{Chermak:2015} states that, in the situation of Theorem~\ref{P:Localities}, $\<\N,\K\>$ is a partial normal subgroup with $\<\N,\K\>\cap S=(\N\cap S)(\K\cap S)$. Theorem~\ref{P:Localities} gives however some more precise information, which is actually needed to prove the theorems below.

\begin{Th}\label{C:NK}
Let $(\L,\Delta,S)$ be a regular locality, $\N\unlhd \L$, $T:=S\cap\N$ and $\K\subn N_\L(T)$. Then $\N\K=\K\N$ is a partial subnormal subgroup of $\L$ and $S\cap\N\K=T(S\cap \K)$.
\end{Th}

The main results for fusion systems are summarized in the following theorem.

\begin{Th}\label{T:ED}
Let $\F$ be a saturated fusion system over $S$, let $\E$ be a normal subsystem of $\F$ over $T$, and let $\mD$ be a subnormal subsystem of  $N_\F(T)$ over $R$. Then there exists a subsystem $\E\mD=(\E\mD)_\F$ of $\F$ such that the following hold:
\begin{itemize}
 \item [(a)] $\E\mD$ is a subnormal subsystem of $\F$ over $TR$ with $\E\unlhd\E\mD$ and $\mD\subn N_{\E\mD}(T)$. Indeed, $\E\mD$ is the smallest subnormal subsystem of $\F$ with these properties.
 \item [(b)] If $\mD$ is normal in $N_\F(T)$, then $\E\mD$ is normal in $\F$ and $\mD$ is normal in $N_{\E\mD}(T)$.
 \item [(c)] Let $\tE$ be a subnormal subsystem of $\F$ with $\E\subn\tE$ and $\mD\subn N_{\tE}(T)$. Then $\E\unlhd\tE$, $\E\mD\subn\tE$ and $\mD\subn N_{\tE}(T)$. If $\mD\unlhd N_\F(T)$, then $\mD\unlhd N_{\tE}(T)$.
 \item [(d)] $N_{\E\mD}(T)=(N_\E(T)\mD)_{N_\F(T)}$.
 \item [(e)] Suppose $(\L,\Delta,S)$ is a regular locality over $\F$. Let $\Psi_\L\colon\fN(\L)\rightarrow\fN(\F)$ and $\hat{\Psi}_{N_\L(T)}\colon N_\L(T)\rightarrow \fN(N_\F(T))$ be the maps from above given by Theorem~A and Theorem~F in \cite{Chermak/Henke}. Set $\N:=\Psi_\L^{-1}(\E)$ and $\K:=\hat{\Psi}_{N_\L(T)}^{-1}(\mD)$. Then $TR=(\N\K)\cap S$, $\N\K\subn\L$ and $\Psi_\L(\N\K)=\E\mD$.
\end{itemize}
\end{Th}

We write here $(\E\mD)_\F$ instead of $\E\mD$ if we want to emphasize that the product subsystem depends not only on $\E$ and $\mD$, but also on $\F$ (cf. Remark~\ref{R:Example}). For the statement of (d) note that $N_\E(T)\unlhd N_\F(T)$ and $\mD\subn N_\F(T)=N_{N_\F(T)}(T)$, so the product subsystem $(N_\E(T)\mD)_{N_\F(T)}$ is defined by the preceding claims in Theorem~\ref{T:ED}. As we state in the following corollary, the statement in (e) can be generalized to arbitrary linking localites if $\mD$ is normal in $N_\F(T)$.

\begin{corollary}\label{C:EDNormalLocalities}
Assume the hypothesis of Theorem~\ref{T:ED}.
Suppose $(\L,\Delta,S)$ is a linking locality over $\F$. Let $\Psi_\L\colon\fN(\L)\rightarrow\fN(\F)$ and $\Psi_{N_\L(T)}\colon N_\L(T)\rightarrow \fN(N_\F(T))$ be the maps from above given by \cite[Theorem~A]{Chermak/Henke}. Set $\N:=\Psi_\L^{-1}(\E)$ and $\K:=\Psi_{N_\L(T)}^{-1}(\mD)$. Then $TR=(\N\K)\cap S$, $\N\K\unlhd\L$ and $\Psi_\L(\N\K)=\E\mD$ is the product subsystem from Theorem~\ref{T:ED}.
\end{corollary}

If $\mD$ is normal in $N_\F(T)$, then we can also give a more concrete description of the product subsystem $\E\mD$. For this we refer the reader to Theorem~\ref{T:ED0}(e).

\section{The proof of Theorem~\ref{P:Localities} and further results on localities}\label{S:Localities}

\subsection{Basic background}

We adapt the terminology and notation of \cite{Chermak:2015} and of \cite{Henke:2015}, which is summarized in \cite[Sections~3.1-3.5]{Henke:Regular}. The reader is referred to these sources for an introduction to partial groups and localities. However, we recall some of the notations and basic results now and along the way.

\smallskip

Suppose $\L$ is a partial group with product $\Pi\colon \D\rightarrow\L$. By $\W(\L)$ we denote the set of words in $\L$. Moreover,  for every $w=(f_1,f_2,\dots,f_n)\in\D$, the product $\Pi(w)$ is also denoted by $f_1f_2\cdots f_n$.

\smallskip

Assume now in addition that $(\L,\Delta,S)$ is a locality. By $S_f$ we denote the set of all $s\in S$ such that $s^f$ is defined and an element of $S$. More generally, if $w=(f_1,\dots,f_n)\in\W(\L)$, then $S_w$ is the set of all $s\in S$ for which there exist $s=s_0,s_1,\dots,s_n\in S$ such that  $s_{i-1}^{f_i}$ is defined and equal to $s_i$ for all $i=1,\dots,n$. We will frequently use that, by \cite[Corollary~2.6]{Chermak:2015}, for every $w\in\W(\L)$, $S_w$ is a subgroup of $S$ and the following equivalence holds:
\begin{equation}\label{E:1}
 S_w\in\Delta\mbox{ if and only if }w\in\D.
\end{equation}
As $f\in\D$ for every $f\in\L$, it follows in particular that $S_f\in\Delta$ for every $f\in\L$. If $P\in\Delta$ with $P\leq S_w$, then we say that $w\in\D$ via $P$.

\smallskip

We will also use the following property which is a consequence of \cite[Lemma~2.3(c)]{Chermak:2015}:
\begin{equation}\label{E:2}
 S_w\leq S_{\Pi(w)}\mbox{ and }(\cdots(S_w^{f_1})^{f_2}\cdots)^{f_n}=S_w^{\Pi(w)}\mbox{ for every }w=(f_1,\dots,f_n)\in\D.
\end{equation}
Using the axioms of a partial group, one sees moreover that
\begin{equation}\label{E:3}
v\circ (\One)\in\D\mbox{ and }\Pi(v)=\Pi(v\circ\emptyset)=\Pi(v\circ(\One))\mbox{ for every }v\in\D. 
\end{equation}

\subsection{The proof of Theorem~\ref{P:Localities}}

\begin{lemma}\label{L:Decompose}
Let $(\L,\Delta,S)$ be a locality with partial product $\Pi\colon\D\rightarrow\L$, $\N\unlhd\L$, $T:=S\cap\N$ and $\K\unlhd N_\L(T)$. Then for every $g\in\N\K$, there exist $n\in\N$ and $k\in\K$ with $(n,k)\in\D$, $g=nk$ and $S_g=S_{(n,k)}$.
\end{lemma}

\begin{proof}
Let $g\in\N\K$. Then there exist $n\in\N$ and $k\in\K$ with $(n,k)\in\D$ and $g=nk$. By the Frattini Lemma and the Splitting Lemma \cite[Corollary~3.11, Lemma~3.12]{Chermak:2015}, there exist $m\in\N$ and $f\in N_\L(T)$ such that $g=mf$ and $S_g=S_{(m,f)}$. Then the first part of \eqref{E:2} yields
\[P:=S_{(n,k)}\leq S_g=S_{(m,f)}.\]
Notice that $P\in\Delta$ by \eqref{E:1} since $(n,k)\in\D$. Hence, using \eqref{E:1} and the second part of \eqref{E:2}, we see that  $u:=(f,f^{-1},m^{-1},n,k)\in\D$ via $P^m$. By \cite[Lemma~1.4(f)]{Chermak:2015} we have $g^{-1}=(mf)^{-1}=f^{-1}m^{-1}$. Hence, it follows from the axioms of a partial group and \eqref{E:3} that we can make the following computations, where the product is defined on all relevant words: 
\[\Pi(u)=\Pi(f,f^{-1}m^{-1},nk)=\Pi(f,g^{-1},g)=\Pi(f,\One)=f\]
and
\[\Pi(u)=\Pi(\One,m^{-1},n,k)=m^{-1}nk=(m^{-1}n)k.\]
Hence, $(m^{-1},n,k)\in\D$ and $f=\Pi(u)=m^{-1}nk=(m^{-1}n)k$. It follows now from \cite[Lemma~1.4(d)]{Chermak:2015} that $v=(m^{-1},n,k,k^{-1})\in\D$. Using that $f,k\in N_\L(T)$, we can therefore conclude that 
\[m^{-1}n=\Pi(v)=(m^{-1}nk)k^{-1}=fk^{-1}\in N_\L(T).\]
As $m,n\in\N$, it follows $m^{-1}n\in N_\N(T)$ and thus $f=(m^{-1}n)k\in N_\N(T)\K$. Notice $N_\N(T)\unlhd N_\L(T)$. As $(N_\L(T),\Delta,S)$ is a locality by \cite[Lemma~2.12]{Chermak:2015}, it follows thus from \cite[Theorem~1]{Henke:2015a} that there exist $n'\in N_\N(T)$ and $k'\in\K$ such that $(n',k')\in\D$, $f=n'k'$ and $S_f=S_{(n',k')}$. Now $S_g=S_{(m,f)}=S_{(m,n',k')}$ and so $(m,n',k')\in\D$ by \eqref{E:1}. Therefore $g=mf=\Pi(m,n',k')=(mn')k'$ where $mn'\in\N$ and $k'\in\K$. Moreover, using \eqref{E:2}, one observes that $S_g=S_{(m,n',k')}\leq S_{(mn',k')}\leq S_{(mn')k'}=S_g$ and thus $S_g=S_{(mn',k')}$. So the assertion holds with $mn'$ and $k'$ in the roles of $n$ and $k$.
\end{proof}

Given two partial groups $\L$ and $\L'$ with products $\Pi\colon\D\rightarrow\L$ and $\Pi'\colon \D'\rightarrow\L'$ respectively, a \emph{projection} from $\L$ to $\L'$ is a map $\beta\colon \L\rightarrow \L'$ such that
\[\D'=\{(f_1\beta,\dots,f_n\beta)\colon (f_1,\dots,f_n)\in\D\}\]
and $\Pi'(f_1\beta,\dots,f_n\beta)=\Pi(f_1,\dots,f_n)\beta$ for every word $(f_1,\dots,f_n)\in\D$ (cf. \cite[Definition~4.4]{Chermak:2015}). As $\L'\subseteq\D'$, every  projection $\beta\colon\L\rightarrow\L'$ is surjective.

\begin{lemma}\label{L:PartialNormalProjection}
Let $(\L,\Delta,S)$ and $(\ov{\L},\ov{\Delta},S)$ be localities and let $\beta\colon \L\rightarrow \ov{\L}$ be a projection. Then $\M\beta\unlhd\ov{\L}$ for every partial normal subgroup $\M$ of $\L$.
\end{lemma}

\begin{proof}
Let $\M\unlhd\L$. By \cite[Lemma~1.14]{Chermak:2015}, $\N:=\ker(\beta)$ is a partial normal subgroup of $\L$. Hence, \cite[Theorem~1]{Henke:2015a} yields  $\M\N\unlhd\L$. It follows from the definition of the kernel and \eqref{E:3} that $\M\beta=(\M\N)\beta$. Hence, the partial subgroup correspondence \cite[Theorem~4.7]{Chermak:2015} yields the assertion.
\end{proof}

\begin{proof}[Proof of Theorem~\ref{P:Localities}]
Part (a) holds by Lemma~\ref{L:Decompose}. It follows from \cite[Lemma~3.2]{Chermak:2015} that $\N\K=\K\N$ and that (a) implies (b). So it remains to show that $\N\K\unlhd\L$ and $\N\K\cap S=T(\K\cap S)$. 

\smallskip

We show first that $\N\K$ is a partial subgroup. If $n\in\N$ and $k\in\K$ with $(n,k)\in\D$, then by \cite[Lemma~1.4(f)]{Chermak:2015}, we have $(nk)^{-1}=k^{-1}n^{-1}\in\K\N=\N\K$, so  $\N\K$ is closed under inversion. Let now $v=(f_1,\dots,f_l)\in\W(\N\K)$. Then by (b), for each $i=1,\dots,l$, there exist $n_i\in\N_i$ and $k_i\in\K_i$ such that $(k_i,n_i)\in\D$, $f_i=k_in_i$ and $S_{f_i}=S_{(k_i,n_i)}$. Then, setting $w:=(k_1,n_1,\dots,k_l,n_l)$, we have $S_w=S_v\in\Delta$ and thus $w\in\D$ by \eqref{E:1}. 
Set $u:=(k_1,\dots,k_n)$. By \cite[Lemma~3.4]{Chermak:2015}, there exists $g\in\N$ such that $u\circ (g)\in\D$ and $\Pi(w)=\Pi(u\circ (g))$. Observe that $\Pi(u)\in\K$ as $\K$ is a partial subgroup. Hence, by the axioms of a partial group, we have $\Pi(v)=\Pi(w)=\Pi(u\circ(g))=\Pi(u)g\in\K\N=\N\K$. This shows that $\N\K$ is a partial subgroup.

\smallskip

Let $\alpha\colon \L\rightarrow \ov{\L}:=\L/\N$ be the natural projection (which exists by \cite[Lemma~3.16, Corollary~4.5]{Chermak:2015}). Then $\alpha$ is a homomorphism of partial groups and so $\alpha|_{N_\L(T)}\colon N_\L(T)\rightarrow \ov{\L}$ is also a homomorphism of partial groups. Indeed, it follows from \cite[Theorem~4.3(b)]{Chermak:2015} that  $\alpha|_{N_\L(T)}$ is actually a projection of partial groups from $N_\L(T)$ to $\ov{\L}$. Therefore, as  $(N_\L(T),\Delta,S)$ is a locality by \cite[Lemma~2.12]{Chermak:2015} and since $\K\unlhd N_\L(T)$, Lemma~\ref{L:PartialNormalProjection}  yields $\K\alpha\unlhd\ov{\L}$. As $\N$ is the kernel of $\alpha$, it follows from \eqref{E:3} that $(\K\N)\alpha=\K\alpha\unlhd\ov{\L}$. We have seen moreover that $\N\K=\K\N$ is a partial subgroup containing $\N$. Using the partial subgroup correspondence \cite[Proposition~4.7]{Chermak:2015}, this implies $\N\K\unlhd\L$.

\smallskip

If $g\in N_{\N\K}(S)$, then it follows from (a) that $g=nk$ for some $n\in\N$, $k\in\K$ with $(n,k)\in\D$ and $S_g=S_{(n,k)}$. As $g\in N_\L(S)$, it follows $S=S_g=S_{(n,k)}$, so $n\in N_\N(S)$ and $k\in N_\K(S)$. Thus $N_{\N\K}(S)\subseteq  N_\N(S)N_\K(S)$. The converse inclusion holds as $N_\L(S)$ is a subgroup and thus closed under products. Hence, $N_{\N\K}(S)=NK$, where $N:=N_\N(S)$ and $K:=\N_\K(S)$ are normal subgroups of the group $N_\L(S)$; to see that $K\unlhd N_\L(S)$ one uses here that  $N_\L(S)\subseteq N_\L(T)$. Hence, $S\cap \N\K=S\cap N_\L(S)\cap \N\K=S\cap N_{\N\K}(S)=S\cap (NK)=(S\cap N)(S\cap K)=(S\cap \N)(S\cap \K)=T(S\cap\K)$.
\end{proof}

\subsection{Restrictions of localities}

Let $(\L^+,\Delta^+,S)$ be a locality with product $\Pi^+\colon\D^+\rightarrow\L^+$. Suppose $\Delta\subseteq\Delta^+$ is overgroup-closed in $S$ and closed under $\F_S(\L^+)$-conjugacy. We set then
\[\L^+|_\Delta:=\{f\in\L^+\colon S_f\in\Delta\}\mbox{ and }\D:=\{w\in\D^+\colon S_w\in\Delta\}.\]
It turns out that $\L^+|_\Delta$ together with $\Pi^+|_\D\colon \D\rightarrow\L^+|_\Delta$ is a partial group, and $(\L^+|_\Delta,\Delta,S)$ is a locality, called the \emph{restriction} of $\L^+$ to $\Delta$.

\smallskip

Let now $\F$ be a saturated fusion system over $S$. Recall that a locality $(\L,\Delta,S)$ over $\F$ is called a \emph{linking locality} if $\F^{cr}\subseteq\Delta$ and, for every $P\in\Delta$, the normalizer $N_\L(P)$ is a group of characteristic $p$. It is shown in \cite[Theorem~A(b)]{Henke:2015} that there is always a linking locality $(\L^s,\F^s,S)$ over $S$ whose object set is the set $\F^s$ of \emph{$\F$-subcentric subgroups} of $S$ (cf. \cite[Definition~1]{Henke:2015}). Moreover, if $(\L,\Delta,S)$ is any linking locality over $\F$, then $\Delta\subseteq\F^s$ by \cite[Proposition~1(b)]{Henke:2015}, and $\L$ is isomorphic to $\L^s|_\Delta$ via an isomorphism which restricts to the identity on $S$ by \cite[Theorem~A(a)]{Henke:2015}. This gives us a way of moving between linking localities with different object sets. We will use this together with the following Lemma to deduce Corollary~\ref{C:EDNormalLocalities} from the statement in Theorem~\ref{T:ED}(e).

\begin{lemma}\label{L:RestrictionLocality}
Let $(\L^+,\Delta^+,S)$ and $(\L,\Delta,S)$ be a localities with $\Delta\subseteq\Delta^+$ and $\L^+|_\Delta=\L$. Write $\Pi^+$ and $\Pi$ for the products on $\L$ and $\L^+$ respectively. Let $\N^+\unlhd\L^+$ and $\K^+\unlhd N_{\L^+}(T)$. Then $\N:=\N^+\cap\L\unlhd\L$ and $\K:=\K^+\cap\L\unlhd N_\L(T)$. Moreover, $\Pi^+(\N^+,\K^+)\cap\L=\Pi(\N,\K)$.
\end{lemma}

\begin{proof}
It is easy to observe that $\N\unlhd\L$. It follows from \cite[Lemma~2.23(b)]{Henke:2020} that $N_{\L^+}(T)\cap\L=N_\L(T)$. Hence, $\K=\K^+\cap N_\L(T)\unlhd N_\L(T)$. 
As $\Pi$ is the restriction of $\Pi^+$, we have $\Pi(\N,\K)\subseteq \Pi^+(\N^+,\K^+)\cap \L$. Let now $f\in\Pi^+(\N^+,\K^+)\cap \L$. By Lemma~\ref{L:Decompose}, there is $n\in\N^+$ and $k\in\K^+$ such that $f=\Pi^+(n,k)$ and $S_f=S_{(n,k)}$. As $f\in\L$, we have $S_f\in\Delta$ by definition of the restriction. Hence, by \eqref{E:1}, $(n,k)\in\D$ and in particular $n,k\in\L$. Hence, $n\in\N$, $k\in\K$ and $f=\Pi(n,k)\in\Pi(\N,\K)$. 
\end{proof}

\section{Some results on linking localities and regular localities}

\subsection{Regular localities}\label{SS:Regular}

The definition of regular localities and the most important results surrounding the concept are originally due to Chermak \cite{ChermakIII}, but we refer to the treatment of the subject in \cite{Henke:Regular}, as there appear to be some gaps in Chermak's proofs.

\smallskip

As remarked before, if $\F$ is a saturated fusion system, then by \cite[Theorem~A]{Henke:2015} there exists a linking locality $(\L^s,\F^s,S)$ over $\F$ whose object set is the set $\F^s$ of subcentric subgroups. Building on earlier work of Chermak \cite{ChermakIII}, we introduced in \cite[Definition~9.17]{Henke:Regular} a certain partial normal subgroup $F^*(\L)$ of $\L$, for every linking locality $(\L,\Delta,S)$ over $\F$. Then the set $\delta(\F)$ is defined as
\begin{equation*}
\delta(\F):=\{P\leq S\colon P\cap F^*(\L^s)\in\F^s\}.
\end{equation*}
It is shown in \cite[Lemma~10.2]{Henke:Regular} that the set $\delta(\F)$ depends only on $\F$ and not on the choice of $\L^s$. Indeed, for every linking locality $(\L,\Delta,S)$, we have
\begin{equation}\label{E:deltaF}
\delta(\F):=\{P\leq S\colon P\cap F^*(\L)\in\F^s\}.
\end{equation}
A linking locality $(\L,\Delta,S)$ is called a \emph{regular locality}, if $\Delta=\delta(\F)$. For every saturated fusion system $\F$, there exists a regular locality over $\F$ (cf. \cite[Lemma~10.4]{Henke:Regular}).

\smallskip

The next theorem states one of the most important properties of regular localities.

\begin{theorem}[{\cite[Corollary~7.9]{ChermakIII}, \cite[Corollary~10.19]{Henke:Regular}}]\label{T:RegularSubnormal}
Let $(\L,\Delta,S)$ be a regular locality and $\H\subn\L$. Then $\F_{S\cap\H}(\H)$ is saturated and $(\H,\delta(\F_{S\cap\H}(\H)),S\cap\H)$ is a regular locality.
\end{theorem}

Let $(\L,\Delta,S)$ be a regular locality. The theorem above leads to a natural definition of components of $\L$ (cf. \cite[Definition~7.9, Definition~11.1]{Henke:Regular}). If $\K_1,\dots,\K_r$ are components of $\L$, then the product $\prod_{i=1}^r\K_i$ does not depend on the order of the factors and is a partial normal subgroup of $F^*(\L)$ (cf. \cite[Proposition~11.7]{Henke:Regular}). The product of all components of $\L$ is denoted by $E(\L)$ and turns out to be a partial normal subgroup of $\L$ (cf. \cite[Lemma~11.13]{Henke:Regular}). We have moreover  $F^*(\L)=E(\L)O_p(\L)$. Note that  Theorem~\ref{T:RegularSubnormal} makes it possible to form $E(\H)$ for every partial subnormal subgroup $\H$ of $\L$.

\subsection{Normalizers of strongly closed subgroups}

The following lemma will be used frequently.

\begin{lemma}[{\cite[Lemma~3.35(b)]{Henke:Regular}}]\label{L:NLT}
If $T$ is a strongly closed subgroup of $\F$, then $(N_\L(T),\Delta,S)$ is a linking locality over $N_\F(T)$.
\end{lemma}

The next goal is to show in Lemma~\ref{L:NLTRegular} below that $N_\L(T)$ can even be given the structure of a  regular locality, provided $(\L,\Delta,S)$ is regular and there exists a normal subsystem of $\F_S(\L)$ over $T$. We first summarize some background that is needed in the proof.

\smallskip

Suppose $S_1,\dots,S_k$ are subgroups of a $p$-group $S$ with $[S_i,S_j]=1$ for $i\neq j$. If, for all $i=1,2,\dots,k$, $\F_i$ is a fusion system over $S_i$ with $S_i\cap \prod_{j\neq i}S_j\leq Z(\F_i)$, then we introduced in \cite[Section~2.5]{Chermak/Henke} a fusion system $\F_1*\F_2*\cdots *\F_k$ over $S_1S_2\cdots S_k$, which can be regarded as a central product of $\F_1,\dots,\F_k$. The notion is a priori slightly different from Aschbacher's definition of a central product \cite[p.14]{Aschbacher:2011}, but of course closely related (cf. \cite[Lemma~2.19(a),(b)]{Chermak/Henke}). Whenever we write $\F_1*\F_2*\cdots *\F_k$ below, then we mean implicitly that this is well-defined, i.e. $[S_i,S_j]=1$ for all $i\neq j$ and $S_i\cap \prod_{j\neq i}S_j\leq Z(\F_i)$ for all $i=1,\dots,k$. We will use the following lemma.

\begin{lemma}[{\cite[Lemma~2.14(g)]{Henke:Regular}}]\label{E:F1F2}
Let $S$ be a $p$-group, and let $\F_i$ be a fusion system over $S_i\leq S$ for $i=1,2$. Suppose $[S_1,S_2]=1$ and $S_1\cap S_2\leq Z(\F_1)\cap Z(\F_2)$. If $P_i\leq S_i$ for $i=1,2$, then $P_1P_2\in(\F_1*\F_2)^s$ if and only if $P_i\in\F_i^s$ for each $i=1,2$.
\end{lemma}

If $\F$ is a fusion system over $S$, $P\leq S$ and $\E$ is a subsystem of $\F$ over $T\leq S$, then we will write $P\cap \E$ for $P\cap T$. In particular, $S\cap \E=T$.

\smallskip

If $(\L,\Delta,S)$ is a linking locality over a saturated fusion system $\F$, then it is stated in \cite[Theorem~E(d)]{Chermak/Henke} that $F^*(\L)$ corresponds to $F^*(\F)$ under the map $\Psi_\L$ introduced in the introduction. As shown in \cite[Lemma~7.21]{Chermak/Henke}, this implies that
\begin{equation}\label{E:deltaFFusion}
\delta(\F)=\{P\leq S\colon P\cap F^*(\F)\in\F^s\}=\{P\leq S\colon P\cap F^*(\F)\in F^*(\F)^s\}.
\end{equation}

\begin{lemma}\label{L:NLTRegular}
Let $\F$ be a saturated fusion system over $S$, and let $\E$ be a normal subsystem over $T$. Then the following hold:
\begin{itemize}
 \item [(a)] $\delta(N_\F(T))=\{Q\leq S\colon QO_p(N_\F(T))\in\delta(\F)\}$.
 \item [(b)] If $(\L,\Delta,S)$ is a regular locality over $\F$, then $(N_\L(T),\delta(N_\F(T)),S)$ is a regular locality over $N_\F(T)$.
\end{itemize}
\end{lemma}

\begin{proof}
\textbf{(a)} By \cite[Lemma~7.13(c), Corollary~7.18]{Chermak/Henke}, $F^*(\F)=E(\E)*E(C_\F(\E))*\F_{O_p(\F)}(O_p(\F))$ and $E(C_\F(\E))=E(N_\F(T))$. Hence, using \cite[Lemma~2.16(c)]{Chermak/Henke}, we can conclude that
\begin{equation}\label{E:FstarF}
F^*(\F)=E(N_\F(T))*(E(\E)*\F_{O_p(\F)}(O_p(\F))).
\end{equation}
By \cite[Theorem~7.10(e)]{Chermak/Henke}, we have also
\begin{equation}\label{E:FstarNFT}
 F^*(N_\F(T))=E(N_\F(T))*\F_{O_p(N_\F(T))}(O_p(N_\F(T))).
\end{equation}
Fix now $P\leq S$ with $O_p(N_\F(T))\leq P$.  Then $P\cap F^*(N_\F(T))=(P\cap E(N_\F(T)))O_p(N_\F(T))$ by \eqref{E:FstarNFT}. As $(S\cap E(\E))O_p(\F)\leq TO_p(\F)\leq O_p(N_\F(T))\leq P$, it follows moreover from \eqref{E:FstarF} that $P\cap F^*(\F)=(P\cap E(N_\F(T)))(S\cap E(\E))O_p(\F)$. Hence, we have the following equivalences:
\begin{eqnarray*}
 P\in\delta(\F) &\Longleftrightarrow & P\cap F^*(\F)\in F^*(\F)^s\;\;\mbox{(by \eqref{E:deltaFFusion})}\\
&\Longleftrightarrow & (P\cap E(N_\F(T)))(S\cap E(\E))O_p(\F)\in F^*(\F)^s\\
&\Longleftrightarrow & P\cap E(N_\F(T))\in E(N_\F(T))^s\;\;\mbox{(by \eqref{E:FstarF} and Lemma~\ref{E:F1F2})}\\
&\Longleftrightarrow & (P\cap E(N_\F(T)))O_p(N_\F(T))\in F^*(N_\F(T))^s\;\;\mbox{(by \eqref{E:FstarNFT} and Lemma~\ref{E:F1F2})}\\
&\Longleftrightarrow & P\cap F^*(N_\F(T))\in F^*(N_\F(T))^s\\
&\Longleftrightarrow & P\in \delta(N_\F(T))\;\;\mbox{(by \eqref{E:deltaFFusion})}.
\end{eqnarray*}
It follows from \cite[Lemma~10.6]{Henke:Regular} that $\delta(N_\F(T))=\{Q\leq S\colon QO_p(N_\F(T))\in\delta(N_\F(T))\}$. So (a) follows since the above equivalences hold for every subgroup $P\leq S$ with $O_p(N_\F(T))\leq P$.

\smallskip

\textbf{(b)} Let now $(\L,\Delta,S)$ be a regular locality over $\F$. By Lemma~\ref{L:NLT}, $(N_\L(T),\Delta,S)$ is a linking locality over $N_\F(T)$. In particular, by \cite[Proposition~5]{Henke:2015}, $N_{N_\L(T)}(O_p(N_\F(T)))=N_\L(T)$. As $\delta(\F)=\Delta$, part (a) gives $\delta(N_\F(T))=\{P\leq S\colon PO_p(N_\F(T))\in\Delta\}$. Now (b) follows from \cite[Lemma~3.28]{Chermak/Henke}.
\end{proof}

\section{The proof of Theorem~\ref{C:NK}}



The following lemma is needed in the proof of Theorem~\ref{C:NK}.

\begin{lemma}\label{E:3prime}
Let $(\L,\Delta,S)$ be a regular locality, $\N\unlhd\L$, $T:=\N\cap S$ and $T^*=F^*(\L)\cap S$. Then
\[N_\N(T)\subseteq N_\L(T^*).\]
\end{lemma}

\begin{proof}
By \cite[Lemma~11.9]{Henke:Regular}, $F^*(\L)=E(\L)O_p(\L)$ and $T^*=(E(\L)\cap S)O_p(\L)$. In \cite[Notation~5.12]{Henke:Regular}, a partial normal subgroup $\N^\perp$ of $\L$ is defined, and it is shown in \cite[Lemma~11.16]{Henke:Regular} that $E(\L)=E(\N)E(\N^\perp)$. Moreover, it follows from \cite[Lemma~11.13]{Henke:Regular} that $E(\N)$ and $E(\N^\perp)$ are normal in $\L$. Hence, by \cite[Theorem~1]{Henke:2015a}, $E(\L)\cap S=(E(\N)\cap S)(E(\N^\perp)\cap S)$ and so
\[T^*=(E(\N)\cap S)(E(\N^\perp)\cap S)O_p(\L).\]
As $E(\N)\unlhd\L$ and $E(\N)\cap S=E(\N)\cap T$, we have moreover
\[N_\N(T)\subseteq N_\L(T)\subseteq N_\L(E(\N)\cap S).\]
It is shown in \cite[Theorem~10.16(e)]{Henke:Regular} that $\N^\perp=C_\L(\N)$.  Thus, \cite[Lemma~3.5]{Henke:Regular} yields $\N\subseteq C_\L(\N^\perp)$. In particular,
\[N_\N(T)\subseteq \N\subseteq C_\L(E(\N^\perp)\cap S).\]
It is moreover shown in \cite[Lemma~3.13]{Henke:Regular} that $\L=N_\L(O_p(\L))$. This implies the assertion.
\end{proof}

Let now $(\L,\Delta,S)$ be a regular locality and set $T^*:=S\cap F^*(\L)$. As $\Delta=\delta(\F)$ and $\F^s$ is overgroup-closed in $S$ by \cite[Proposition~3.3]{Henke:2015}, it follows from \eqref{E:deltaF} that, for every subgroup $P$ of $S$, we have $P\in\Delta$ if and only if $P\cap T^*=P\cap F^*(\L)\in\Delta$. In particular, it follows from \eqref{E:1} that for every $w\in\W(\L)$ the following equivalence holds:
\begin{equation}\label{E:SwcapT*}
S_w\cap T^*\in\Delta\mbox{ if and only if }w\in\D.
\end{equation}

\begin{proof}[Proof of Theorem~\ref{C:NK}]
Suppose the assertion is false, and let $(\L,\Delta,S,\N,\K)$ be a counterexample with $|\L|-|\K|$ minimal. Set $\F:=\F_S(\L)$ and $T:=S\cap \N$ so that $\K\subn N_\L(T)$. Write $\Pi\colon\D\rightarrow \L$ for the partial product on $\L$. Fix a subnormal series
\[\K=\K_0\unlhd\K_1\unlhd\cdots\unlhd \K_m:=N_\L(T)\]
of minimal length $m$ and set $\tK:=\K_{m-1}$.

\smallskip

By follows from \cite[Lemma~3.2]{Chermak:2015} that $\N\K=\K\N$. Hence, to obtain a contradiction, it is sufficient to prove that $\N\K$ is subnormal in $\L$ and $S\cap (\N\K)=T(S\cap\K)$.

\smallskip

Since $\tK\unlhd N_\L(T)$, it follows from Theorem~\ref{P:Localities} that $\N\tK\unlhd\L$. As $(\L,\Delta,S,\N,\K)$ is a counterexample, it follows in particular that $\K\neq\tK$ and thus $m\geq 2$. Moreover, by Theorem~\ref{T:RegularSubnormal}, $(\N\tK,\Gamma,S\cap (\N\tK))$ forms itself a regular locality for some appropriate set $\Gamma$ of subgroups of $S\cap (\N\tK)$.

\smallskip

Notice that $\N\unlhd\N\tK$. As $N_{\N\tK}(T)=\N\tK\cap N_\L(T)$ is a partial subgroup of $N_\L(T)$ containing $\K$, it follows moreover from \cite[Lemma~3.7(a)]{Henke:Regular} that $\K\subn N_{\N\tK}(T)$. Hence, if $\N\tK\neq \L$, then the minimality of $|\L|-|\K|$ yields that $(\N\tK,\Gamma,S\cap(\N\tK),\N,\K)$ is not a counterexample. This means that $S\cap (\N\K)=(S\cap(\N\tK))\cap (\N\K)=T(S\cap\K)$ and $\N\K\subn\N\tK\unlhd\L$. In particular, $\N\K$ is then subnormal in $\L$, so we obtain a contradiction to our assumption. Therefore we have shown that
\begin{equation}\label{E:R1}
\L=\N\tK.
\end{equation}
By \cite[Theorem~A]{Chermak/Henke}, $\F_T(\N)\unlhd\F$. Thus, Lemma~\ref{L:NLTRegular}(b) gives that $(N_\L(T),\delta(N_\F(T)),S)$ is a regular locality. As $\tK\unlhd N_\L(T)$, it follows thus from Theorem~\ref{T:RegularSubnormal} that $(\tK,\tGamma,S\cap\tK)$ forms a regular locality for some set $\tGamma$ of subgroups of $S\cap \tK$. Note that $\N\cap\tK\unlhd\tK$. Moreover, $T_0:=S\cap \N\cap\tK=T\cap\tK$. As $\tK\subseteq N_\L(T)$, it follows $\tK=N_{\tK}(T_0)$. In particular, $\K\subn N_{\tK}(T_0)$. The minimality of $m$ yields $\tK\neq N_\L(T)$ and thus $|\tK|<|\L|$. Hence, it follows from the minimality of $|\L|-|\K|$ that $(\tK,\tGamma,S\cap\tK,\N\cap\tK,\K)$ is not a counterexample. This means that
\[\K':=(\N\cap \tK)\K\subn\tK\mbox{ and }S\cap \K'=(S\cap\tK)\cap \K'=T_0(S\cap\K).\]
As $\tK\unlhd N_\L(T)$, we obtain therefore
\begin{equation}\label{E:R2}
\K'\subn N_\L(T)\mbox{ and }S\cap\K'=T_0(S\cap\K).
\end{equation}
  We show next that
\begin{equation}\label{E:R3}
 \N\K=\N\K'.
\end{equation}
Clearly, we have $\K\subseteq\K'$ and thus $\N\K\subseteq\N\K'$. Let now $n\in\N$ and $f\in\K'$ with $(n,f)\in\D$. We need to show that $nf\in\N\K$. As $f\in\K'=(\N\cap\tK)\K$, there exist $\tn\in\N\cap\tK$ and $k\in\K$ with $(\tn,k)\in\D$ and $f=\tn k$. It is indeed enough to show that $(n,\tn,k)\in\D$, since then the ``associativity axiom'' of partial groups yields $nf=n(\tn k)=\Pi(n,\tn,k)=(n\tn)k\in\N\K$ as required.

\smallskip

To prove $(n,\tn,k)\in\D$, we note first that, by \cite[Lemma~1.4(d)]{Chermak:2015} and the axioms of a partial group, $(\tn^{-1},\tn,k)\in\D$ and $k=\Pi(\tn^{-1},\tn,k)=\tn^{-1}(\tn k)=\tn^{-1}f$. Moreover, by Lemma~\ref{E:3prime}, we have $\tn\in\N\cap\tK\subseteq N_\N(T)\subseteq N_\L(T^*)$ and thus $S_f\cap T^*\leq S_{(\tn,\tn^{-1},f)}$. Using these properties along with \eqref{E:2}, we see that
\[S_{(\tn,k)}\cap T^*\leq S_f\cap T^*\leq S_{(\tn,\tn^{-1},f)}\cap T^*\leq S_{(\tn,\tn^{-1}f)}\cap T^*=S_{(\tn,k)}\cap T^*\]
and thus $S_{(\tn,k)}\cap T^*=S_f\cap T^*$. In particular, $S_{(n,f)}\cap T^*=S_{(n,\tn,k)}\cap T^*$. As $(n,f)\in\D$, it follows from \eqref{E:SwcapT*} first that $S_{(n,\tn,k)}\cap T^*=S_{(n,f)}\cap T^*\in\Delta$ and then $(n,\tn,k)\in\D$. As argued before this shows \eqref{E:R3}.

\smallskip

Assume now that $\K$ is properly contained in $\K'$. The minimality of $|\L|-|\K|$ yields then that $(\L,\Delta,S,\N,\K')$ is not a counterexample. As $\K'\subn N_\L(T)$  and $S\cap \K'=T_0(S\cap\K)$ by \eqref{E:R2}, it follows thus from \eqref{E:R3} that $\N\K=\N\K'\subn\L$ and $S\cap (\N\K)=S\cap (\N\K')=T(S\cap\K')=TT_0(S\cap\K)=T(S\cap\K)$.  This contradicts the assumption that $(\L,\Delta,S,\N,\K)$ is a counterexample. Hence, we have shown that $\K=\K'$ and thus
\begin{equation}\label{E:R4}
 \N\cap\tK\subseteq\K.
\end{equation}
Recall that $m\geq 2$ and thus it makes sense to consider $\K_{m-2}$. We show next that
\begin{equation}\label{E:R5}
 x^n\in\K_{m-2}\mbox{ for all }x\in\K_{m-2}\mbox{ and }n\in N_\N(T)\mbox{ with }(n^{-1},x,n)\in\D.
\end{equation}
For the proof fix $x\in\K_{m-2}$ and $n\in N_\N(T)$ with $(n^{-1},x,n)\in\D$. Using \cite[Lemma~3.2(a)]{Chermak:2015} (with $(n^{-1},x)$ in place of $(x,f)$), we observe that  $P:=S_{(n^{-1},x,n)}\leq S_{(n^{-1},x)}\leq S_x$ and thus $v:=(x^{-1},n^{-1},x,n)\in\D$ via $P^x$. As $\tK\unlhd N_\L(T)$ and $\N\unlhd \L$, we see then using \eqref{E:R4} that
\[x^{-1}x^n=\Pi(v)=(n^{-1})^xn\in\tK\cap \N\subseteq\K\subseteq\K_{m-2}.\]
As $(x^{-1},x^n)\in\D$, \cite[Lemma~1.4(d)]{Chermak:2015} allows us to conclude that $(x,x^{-1},x^n)\in\D$ and $x^n=x(x^{-1}x^n)\in \K_{m-2}$. This proves \eqref{E:R5}.

\smallskip

We will now be able to obtain a contradiction to the minimality of $m$ by showing
\begin{equation}\label{E:R6}
 \K_{m-2}\unlhd N_\L(T).
\end{equation}
For the proof fix $x\in \K_{m-2}$ and $g\in N_\L(T)$ with $(g^{-1},x,g)\in\D$. By \eqref{E:R1} and Theorem~\ref{P:Localities}, there exist $n\in\N$ and $k\in\tK$ with $(n,k)\in\D$, $g=nk$ and $S_g=S_{(n,k)}$. It follows then from \cite[Lemma~1.4(f)]{Chermak:2015} that $g^{-1}=k^{-1}n^{-1}$ and from \cite[Lemma~2.3(c), Proposition~2.5(b)]{Chermak:2015} that $S_{g^{-1}}=S_{(k^{-1},n^{-1})}$. Hence, $w:=(k^{-1},n^{-1},x,n,k)\in\D$ via $S_{(g^{-1},x,g)}$ and so
\[x^g=\Pi(g^{-1},x,g)=\Pi(w)=(x^n)^k.\]
As $T\leq S_g=S_{(n,k)}\leq S_n$ and $T$ is strongly closed, we have $n\in N_\N(T)$. Hence, it follows from \eqref{E:R5} that $x^n\in\K_{m-2}$. As $\K_{m-2}\unlhd \K_{m-1}=\tK$ and $k\in\tK$, we can conclude that $x^g=(x^n)^k\in\K_{m-2}$. This shows \eqref{E:R6}. As noticed before, this  contradicts the minimality of $m$.
\end{proof}

\section{Products in fusion systems}


\subsection{Reminder of some background}\label{SS:PsiL}

If $\F$ is a saturated fusion system, then as before, we  write $\fN(\F)$ for the set of normal subsystems of $\F$, and $\fS(\F)$ for the set of subnormal subsystems of $\F$. Similarly, if $\L$ is a partial group, then $\fN(\L)$ denotes the set of partial normal subgroups of $\L$, and $\fS(\L)$ denotes the set of partial subnormal subgroups of $\L$.

\smallskip

Let now $(\L,\Delta,S)$ be a linking locality over $\F$. Then by \cite[Theorem~A]{Chermak/Henke}, there is a bijection
\[\Psi_\L\colon\fN(\L)\rightarrow\fN(\F)\]
which sends a partial normal subgroup $\N$ to the smallest normal subsystem of $\F$ over $S\cap\N$ containing $\F_{S\cap\N}(\N)$. If $T\leq S$ is strongly $\F$-closed, then $(N_\L(T),\Delta,S)$ is a linking locality over $N_\F(T)$ by Lemma~\ref{L:NLT}. Thus, there exists a bijection
\[\Psi_{N_\L(T)}\colon N_\L(T)\rightarrow N_\F(T)\]
which sends a partial normal subgroup $\K$ of $N_\L(T)$ to the smallest normal subsystem over $S\cap\K$ containing $\F_{S\cap\K}(\K)$.

\smallskip

Suppose now that $(\L,\Delta,S)$ is a regular locality. Then the map $\Psi_\L$ from above is given by $\Psi_\L(\N)=\F_{S\cap\N}(\N)$. Moreover, by \cite[Theorem~F]{Chermak/Henke}, the bijection $\Psi_\L$ extends to a bijection
\[\hat{\Psi}_\L\colon\fS(\L)\rightarrow \fS(\F),\H\mapsto\F_{S\cap\H}(\H).\]
If $\E$ is a normal subsystem of $\F$ over $T$, we saw furthermore in Lemma~\ref{L:NLTRegular} that $(N_\L(T),\delta(N_\F(T)),S)$ is a regular locality over $N_\F(T)$. Hence, if a normal subsystem of $\F$ over $T$ exists, then the map $\Psi_{N_\L(T)}$ from above extends to the bijection
\[\hat{\Psi}_{N_\L(T)}\colon\fS(N_\L(T))\rightarrow \fS(N_\F(T)),\K\mapsto \F_{S\cap\K}(\K).\]

\subsection{The proof of Theorem~\ref{T:ED} and some related properties}

To prove Theorem~\ref{T:ED}, we will use the existence of a regular locality over $\F$ and then apply Theorems~\ref{P:Localities} and \ref{C:NK}. We need the following elementary lemma.

\begin{lemma}\label{L:NET}
Let $(\L,\Delta,S)$ be a locality, let $\H$ be a partial subgroup of $\L$, and set $\E_\H:=\F_{S\cap\H}(\H)$. Let $T\leq S\cap\H$ be such that $T$ is strongly $\E_\H$-closed. Then $S\cap\H=S\cap N_\H(T)$ and
\[N_{\E_\H}(T)=\F_{S\cap \H}(N_\H(T)).\]
\end{lemma}

\begin{proof}
Observe first that $T\unlhd S\cap\H$ as $T$ is strongly $\E_\H$-closed. Thus, $N_{\E_\H}(T)$ is a fusion system over $S\cap\H=S\cap N_\H(T)$. Clearly $\F_{S\cap\H}(N_\H(T))\subseteq N_{\E_\H}(T)$. On the other hand, if we consider a morphism $\phi$ in $N_{\E_\H}(T)$, we can assume that $\phi$ is defined on $T$. By definition of $\E_\H$, $\phi=(c_{f_1}|_{P_1})\circ (c_{f_2}|_{P_2})\circ\cdots\circ (c_{f_k}|_{P_k})$ for some $f_1,f_2,\dots,f_k\in\H$ and $P_i\leq S_{f_i}\cap\H$ for $i=1,\dots,k$. Then $T\leq P_1$ and, as $T$ is strongly $\E_\H$-closed, it follows inductively that $T\leq P_i$ and $f_i\in N_\H(T)$ for $i=1,\dots,k$. So $\phi$ is a morphism in $\F_{S\cap\H}(N_\H(T))$ and the assertion holds.
\end{proof}

We restate and prove Theorem~\ref{T:ED} in Theorem~\ref{T:ED0} below. Apart from that, in Theorem~\ref{T:ED0}(e), we will give a concrete description of $\E\mD$ in the case that $\mD$ is normal in $N_\F(T)$. We use that, for a normal subsystem $\E$ of $\F$ and a subgroup $R\leq S$, a product subsystem $\E R$ is defined as follows.

\begin{definition}\label{D:Product}
Let $\E$ be a normal subsystem of $\F$ over $T\leq S$. For every $P\leq S$, set
\begin{eqnarray*}
\Ac_{\F,\E}(P):=\<\phi\in\Aut_\F(P)\colon \phi\mbox{ $p^\prime$-element},\;[P,\phi]\leq P\cap T\mbox{ and }\phi|_{P\cap T}\in\Aut_\E(P\cap T)\>.
\end{eqnarray*}
For every subgroup $R$ of $S$, define
\begin{eqnarray*}
\E R:=(\E R)_\F:=\<\Ac_{\F,\E}(P)\colon P\leq TR\mbox{ and }P\cap T\in\E^c\>_{TR}
\end{eqnarray*}
and call $\E R=(\E R)_\F$ the \emph{product} of $\E$ with $R$ (formed inside of $\F$).
\end{definition}

A product subsystem $\E R$ was first introduced by Aschbacher \cite[Chapter~8]{Aschbacher:2011}. The construction above was given by the author of this paper \cite{Henke:2013}. As shown in \cite[Example~7.4]{Henke:2013}, the product $\E R=(\E R)_\F$ depends actually not only on $\E$, $R$ and $S$, but also on $\F$.

\smallskip

Note that Theorem~\ref{T:ED} is implied by the following theorem.

\begin{theorem}\label{T:ED0}
Let $\F$ be a saturated fusion system over $S$, let $\E$ be a normal subsystem of $\F$ over $T$, and let $\mD$ be a subnormal subsystem of $N_\F(T)$ over $R$. Then there exists a subsystem $\E\mD=(\E\mD)_\F$ of $\F$ such that the following hold:
\begin{itemize}
 \item [(a)] $\E\mD$ is a subnormal subsystem of $\F$ over $TR$ with $\E\unlhd\E\mD$ and $\mD\subn N_{\E\mD}(T)$. If $\mD\unlhd N_\F(T)$, then $\E\mD\unlhd\F$ and $\mD\unlhd N_{\E\mD}(T)$.
 \item [(b)] $\E\mD$ is the smallest subnormal subsystem $\tE$ of $\F$ with $\E\unlhd\tE$ and $\mD\subn N_{\tE}(T)$. Indeed, if $\tE\subn\F$ with $\E\subn\tE$ and $\mD\subn N_{\tE}(T)$, then $\E\mD$ is a subnormal subsystem of $\tE$, $\E\unlhd\tE$ and $\mD\subn N_{\tE}(T)$. If $\mD\unlhd N_\F(T)$, then $\E\mD\unlhd\tE$ and $\mD\unlhd N_{\tE}(T)$.
 \item [(c)] $N_{\E\mD}(T)$ equals the product $(N_\E(T)\mD)_{N_\F(T)}$.
 \item [(d)] Let $(\L,\Delta,S)$ be a regular locality over $\F$, and suppose $\Psi_\L$ and $\hat{\Psi}_{N_\L(T)}$ are the maps from above (given by Theorem~A and Theorem~F in \cite{Chermak/Henke}). Setting $\N:=\Psi_\L^{-1}(\E)$ and $\K:=\hat{\Psi}_{N_\L(T)}^{-1}(\mD)$, we have then $TR=(\N\K)\cap S$ and $\Psi_\L(\N\K)=\E\mD$.
 \item [(e)] If $\mD\unlhd N_\F(T)$, then
  \[\E\mD=\<(\E R)_\F,(\mD T)_{N_\F(T)}\>\] is the subsystem over $TR$ which is generated by the automorphism groups $\Ac_{\F,\E}(P)$ where $P\leq TR$ with $P\cap T\in\E^c$ and the automorphism groups $\Ac_{N_\F(T),\mD}(Q)$ where $Q\leq TR$ with $Q\cap R\in\mD^c$.
\end{itemize}
\end{theorem}

As mentioned in the introduction, if $\E$ and $\mD$ are as in Theorem~\ref{T:ED0}, then $N_\E(T)\unlhd N_\F(T)$ and $\mD\subn N_\F(T)=N_{N_\F(T)}(T)$. Thus, by the first part of Theorem~\ref{T:ED0}, the product subsystem $(N_\E(T)\mD)_{N_\F(T)}$ is defined. Therefore, the statement in  Theorem~\ref{T:ED0}(c) makes sense.

\begin{proof}[Proof of Theorem~\ref{T:ED0}]
Let $(\L,\Delta,S)$ be a regular locality over $S$ (which exists by \cite[Lemma~10.4]{Henke:Regular}). Then $(N_\L(T),\delta(N_\F(T)),S)$ is a regular locality over $N_\F(T)$ by Lemma~\ref{L:NLTRegular}(b). We will consider the maps $\Psi_\L$, $\hat{\Psi}_\L$, $\Psi_{N_\L(T)}$ and $\hat{\Psi}_{N_\L(T)}$ given by \cite[Theorem~A, Theorem~F]{Chermak/Henke} (which we introduced in Subsection~\ref{SS:PsiL}). Set
\[\N:=\Psi_\L^{-1}(\E)\unlhd\L\mbox{ and }\K:=\hat{\Psi}_{N_\L(T)}^{-1}(\mD)\subn N_\L(T).\]
Note that $S\cap\K=R$ by definition of $\hat{\Psi}_{N_\L(T)}$. By Theorem~\ref{C:NK},
\[\N\K=\K\N\subn\L \mbox{ with }(\N\K)\cap S=T(S\cap\K)=TR.\]
Hence, 
\[\E\mD:=\hat{\Psi}_\L(\N\K)=\F_{TR}(\N\K)\]
is well-defined and a subnormal subsystem over $TR$.

\smallskip

\textbf{(a)} Note that $\N\unlhd\N\K$ as $\N\unlhd\L$ and $\N\subseteq\N\K$. Hence, it follows from \cite[Proposition~7.1(c)]{Chermak/Henke}, that
\[\E=\Psi_\L(\N)=\hat{\Psi}_\L(\N)\unlhd\hat{\Psi}_\L(\N\K)=\E\mD.\]
Notice $N_{\N\K}(T)=\N\K\cap N_\L(T)\subn N_\L(T)$ as $\N\K\subn\L$ and $N_{\N\K}(T)\cap S=(\N\K)\cap S=TR$. Moreover, $N_{\E\mD}(T)=\F_{TR}(N_{\N\K}(T))=\hat{\Psi}_{\N_\L(T)}(N_{\N\K}(T))$, where the first equality uses Lemma~\ref{L:NET} with $\N\K$ in place of $\H$. Hence, $\K\subseteq N_{\N\K}(T)$ implies by \cite[Proposition~7.1(c)]{Chermak/Henke} that $\mD=\hat{\Psi}_{N_\L(T)}(\K)\subn \hat{\Psi}_{N_\L(T)}(N_{\N\K}(T))=N_{\E\mD}(T)$. Hence (a) holds.

\smallskip

\textbf{(b)} Let $\tE\subn\F$ with $\E\subn\tE$ and $\mD\subn N_{\tE}(T)$. Recall that the bijections $\hat{\Psi}_\L$ and
and $\hat{\Psi}_{N_\L(T)}$ restrict to $\Psi_\L$ and $\Psi_{N_\L(T)}$ respectively. Set $\H:=\hat{\Psi}_\L^{-1}(\tE)$. Then $\H\subn\L$ and $\tE=\F_{S\cap\H}(\H)$. Observe that $N_\H(T)\subn N_\L(T)$, since a subnormal series $\H=\H_0\unlhd\H_1\unlhd\cdots\unlhd \H_n=\L$ of $\H$ in $\L$ leads to a subnormal series
\[N_\H(T)=\H_0\cap N_\L(T)\unlhd \H_1\cap N_\L(T)\unlhd\cdots\unlhd \H_n\cap N_\L(T)=N_\L(T).\]
Thus, by Lemma~\ref{L:NET},
\[N_{\tE}(T)=\F_{S\cap\H}(N_\H(T))=\hat{\Psi}_{N_\L(T)}(N_\H(T))\subn\N_\F(T).\]
As $\hat{\Psi}_\L(\N)=\Psi_\L(\N)=\E\subn\tE$ and $\hat{\Psi}_{N_\L(T)}(\K)=\mD\subn N_{\tE}(T)$ by assumption, it follow now from \cite[Proposition~7.1(b)]{Chermak/Henke} that $\N\subseteq \H$ and $\K\subseteq N_\H(T)$. In particular, $\N\K\subseteq\H$. Observe that indeed $\N\unlhd\H$, $\N\K\subn\H$ and $\K\subn N_\H(T)$ by \cite[Lemma~3.7(a)]{Henke:Regular}. Hence, it follows now from \cite[Proposition~7.1(c)]{Chermak/Henke} that $\E\unlhd\tE$, $\E\mD\subn\tE$ and $\mD\subn N_{\tE}(T)$.

\smallskip

If $\mD\unlhd N_\F(T)$, then $\K=\Psi_{N_\L(T)}^{-1}(\mD)\unlhd N_\L(T)$ and thus $\K\unlhd N_\H(T)$. Moreover, $\N\K\unlhd\L$ by Theorem~\ref{P:Localities} and so $\N\K\unlhd\H$. Hence, \cite[Proposition~7.1(c)]{Chermak/Henke} yields in this case that $\E\mD\unlhd\tE$ and $\mD\unlhd N_{\tE}(T)$.

\smallskip

\textbf{(d)} It follows from part (b) that the subsystem $\E\mD$ depends only on $\F$, $\E$ and $\mD$ and not on the choice of the regular locality $(\L,\Delta,S)$. Hence, (d) follows from the definition of $\E\mD$ above.

\smallskip

\textbf{(c)} By the Dedekind Lemma for partial groups \cite[Lemma~1.10]{Chermak:2015}, $N_{\N\K}(T)=N_\N(T)\K$. Notice moreover that $N_\N(T)\unlhd N_\L(T)$ and, by Lemma~\ref{L:NET} applied with $\N$ in place of $\H$, $N_\E(T)=\F_T(N_\N(T))=\Psi_{N_\L(T)}(N_\N(T))\unlhd N_\F(T)$. As seen in the proof of (a), we have $N_{\E\mD}(T)=\hat{\Psi}_{N_\L(T)}(N_{\N\K}(T))$. Hence, $N_{\E\mD}(T)=\hat{\Psi}_{N_\L(T)}(N_\N(T)\K)$. As $\K\subn N_\L(T)=N_{N_\L(T)}(T)$ and $(N_\L(T),\delta(N_\F(T)),S)$ is a regular locality over $N_\F(T)$, it follows thus from (d) that $N_{\E\mD}(T)=(N_\E(T)\mD)_{N_\F(T)}$.

\smallskip

\textbf{(e)} By definition of $(\E R)_\F$ and $(\mD T)_{N_\F(T)}$ (cf. Definition~\ref{D:Product}), it is for the proof of (e) sufficient to show $\E\mD=\<(\E R)_\F,(\mD T)_{N_\F(T)}\>$. As $(\L,\Delta,S)$ and $(N_\L(T),\delta(N_\F(T)),S)$ are regular localities over $\F$ and $N_\F(T)$ respectively, it follows from \cite[Proposition~8.2]{Chermak/Henke} that $(\E R)_\F=\F_{TR}(\N R)$ and $(\mD T)_{N_\F(T)}=\F_{TR}(\K T)$. Hence, we only need to show that $\E\mD=\<\F_{TR}(\N R),\F_{TR}(\K T)\>$. Recall that $\E\mD=\F_{TR}(\N\K)$, $T=S\cap\N$ and $R=S\cap \K$. So it is clear that $\<\F_{TR}(\N R),\F_{TR}(\K T)\>\subseteq\E\mD$. Moreover, since $\E\mD$ is generated by morphisms of the form $c_f|_{S_f\cap(\N\K)}$ with $f\in\N\K$, we just need to argue that each such morphism can be written as the composite of a morphism in $\F_{TR}(\N T)$ and a morphism in $\F_{TR}(\K R)$. By Lemma~\ref{L:Decompose}, for each $f\in\N\K$ there are elements $n\in\N$ and $k\in\K$ such that $(n,k)\in\D$, $S_f=S_{(n,k)}$ and $f=nk$, which implies that $c_f|_{S_f\cap (\N\K)}=(c_n|_{S_f\cap (\N\K)})\circ (c_k|_{S_f^n\cap (\N\K)})$ is of the required form.
\end{proof}

\begin{remark}\label{R:Example}
In \cite[Example~7.4]{Henke:2013} we constructed two saturated fusion systems $\F$ and $\G$ over the same $p$-group $S$ such that $\F\neq\G$, $\E:=O^p(\F)=O^p(\G)$, and $S$ is normal in $\F$ and $\G$. It is a consequence of \cite[Theorem~1]{Henke:2013} that $(\E S)_\F=(O^p(\F)S)_\F=\F$ and similarly $(\E S)_\G=\G$.

\smallskip

Fix now $T\leq S$ such that $\E$ is a fusion system over $T$. As $S$ is normal in $\F$ and $\G$, it follows in particular that $\mD:=\F_S(S)$ is normal in $N_\F(T)$ and $N_\G(T)$. Hence, by Theorem~\ref{T:ED0}(e), we have $(\E\mD)_\F=\<(\E S)_\F,\F_S(S)\>=(\E S)_\F=\F$ and similarly $(\E\mD)_\G=\G$. This shows that the product $(\E\mD)_\F$ depends actually not only on $\E$ and $\mD$, but also on $\F$.
\end{remark}

\subsection{The proof of Corollary~\ref{C:EDNormalLocalities}}

\begin{lemma}\label{L:RestrictionPsiL}
 Let $(\L,\Delta,S)$ and $(\L^+,\Delta^+,S)$ be linking localities over $\F$ with $\L^+|_\Delta=\L$. Let $\E$ be a normal subsystem of $\F$ over $T$, and let $\mD$ be a normal subsystem of $N_\F(T)$ over $R$. Set $\N:=\Psi_\L^{-1}(\E)$, $\K:=\Psi_{N_\L(T)}^{-1}(\mD)$, $\N^+:=\Psi_{\L^+}^{-1}(\E)$ and $\K^+:=\Psi_{N_{\L^+}(T)}^{-1}(\mD)$. Then
 \[\Psi_{\L^+}(\N^+\K^+)=\Psi_\L(\N\K).\]
\end{lemma}

\begin{proof}
It follows from \cite[Lemma~2.23(b)]{Henke:2020} that $N_{\L^+}(T)\cap\L=N_{\L^+}(T)|_\Delta=N_\L(T)$.
Define
\[\Phi_{\L^+,\L}\colon \fN(\L^+)\rightarrow \fN(\L),\;\M^+\mapsto \M^+\cap\L\]
and similarly
\[\Phi_{N_{\L^+}(T),N_\L(T)}\colon\fN(N_{\L^+}(T))\rightarrow \fN(N_\L(T)),\M^+\mapsto\M^+\cap N_\L(T).\]
By \cite[Theorem~5.14(c)]{Chermak/Henke}, we have $\Psi_{\L^+}=\Psi_\L\circ\Phi_{\L^+,\L}$ and $\Psi_{N_{\L^+}(T)}=\Psi_{N_\L(T)}\circ \Phi_{N_{\L^+}(T),N_\L(T)}$. Hence, $\N^+\cap\L=\Phi_{\L^+,\L}(\N^+)=\Psi_\L^{-1}(\Psi_{\L^+}(\N^+))=\Psi_\L^{-1}(\E)=\N$ and $\K^+\cap \L=\K^+\cap N_\L(T)=\Phi_{N_{\L^+}(T),N_\L(T)}(\K^+)=\Psi_{N_\L(T)}^{-1}(\Psi_{N_{\L^+}(T)}(\K^+))=\Psi_{N_\L(T)}^{-1}(\mD)=\K$. Hence, by Lemma~\ref{L:RestrictionLocality}, we have $\Phi_{\L^+,\L}(\N^+\K^+)=\N^+\K^+\cap\L=\N\K$ and thus $\Psi_\L(\N\K)=\Psi_{\L^+}(\N^+\K^+)$. 
\end{proof}

\begin{proof}[Proof of Corollary~\ref{C:EDNormalLocalities}]
Let $(\L,\Delta,S)$ be a linking locality. By \cite[Proposition~3.3, Theorem~7.2]{Henke:2015}, there exists a linking locality $(\L^s,\F^s,S)$ over $\F$ with $\L^s|_\Delta=\L$. By \cite[Lemma~10.4]{Henke:Regular},  $\F^{cr}\subseteq\delta(\F)\subseteq \F^s$ and $\delta(\F)$ is closed under $\F$-conjugacy and overgroup-closed in $S$. Hence, $\L_\delta:=\L^s|_{\delta(\F)}$ is well-defined and a locality with object set $\delta(\F)$. As $\L^s$ is a linking locality and $N_{\L_\delta}(P)=N_{\L^s}(P)$ for all $P\in\delta(\F)$, it follows that $(\L_\delta,\delta(\F),S)$ is a linking locality and thus a regular locality.

\smallskip

Let now $\N:=\Psi_\L^{-1}(\E)$, $\N^s:=\Psi_{\L^s}^{-1}(\E)$, $\N_\delta=\Psi_{\L_\delta}^{-1}(\E)$, $\K:=\Psi_{N_\L(T)}^{-1}(\mD)$, $\K^s:=\Psi_{N_{\L^s}(T)}^{-1}(\mD)$, $\K_\delta:=\Psi_{N_{\L_\delta}(T)}^{-1}(\mD)$. Then in particular $\N\cap S=T$ and $\K\cap S=R$. Note that $\N\K\unlhd\L$, $(\N\K)\cap S=TR$, $\N^s\K^s\unlhd\L^s$ and $\N_\delta\K_\delta\unlhd\L_\delta$ by Theorem~\ref{P:Localities}. By Theorem~\ref{T:ED0}(d), we have $\Psi_{\L_\delta}(\N_\delta\K_\delta)=\E\mD$. Hence, applying Lemma~\ref{L:RestrictionPsiL} twice, we obtain $\Psi_\L(\N\K)=\Psi_{\L^s}(\N^s\K^s)=\Psi_{\L_\delta}(\N_\delta\K_\delta)=\E\mD$. This proves the assertion.
\end{proof}

\bibliographystyle{amsalpha}
\bibliography{repcoh}

\newcommand{\etalchar}[1]{$^{#1}$}
\providecommand{\bysame}{\leavevmode\hbox to3em{\hrulefill}\thinspace}
\providecommand{\MR}{\relax\ifhmode\unskip\space\fi MR }
\providecommand{\MRhref}[2]{%
  \href{http://www.ams.org/mathscinet-getitem?mr=#1}{#2}
}
\providecommand{\href}[2]{#2}
\begin{thebibliography}{BCG{\etalchar{+}}05}

\bibitem[AKO11]{Aschbacher/Kessar/Oliver:2011}
M.~Aschbacher, R.~Kessar, and B.~Oliver, \emph{{Fusion systems in algebra and
  topology}}, London Math.\ Soc.\ Lecture Note Series, vol. 391, Cambridge
  University Press, 2011.

\bibitem[Asc08]{Aschbacher:2008}
M.~Aschbacher, \emph{Normal subsystems of fusion systems}, Proc. Lond. Math.
  Soc. (3) \textbf{97} (2008), no.~1, 239--271.

\bibitem[Asc11]{Aschbacher:2011}
\bysame, \emph{The generalized {F}itting subsystem of a fusion system}, Mem.
  Amer. Math. Soc. \textbf{209} (2011), no.~986, vi+110.

\bibitem[BCG{\etalchar{+}}05]{BCGLO1}
C.~Broto, N.~Castellana, J.~Grodal, R.~Levi, and B.~Oliver, \emph{Subgroup
  families controlling {$p$}-local finite groups}, Proc. London Math. Soc. (3)
  \textbf{91} (2005), no.~2, 325--354.

\bibitem[BLO03]{BLO2}
C.~Broto, R.~Levi, and B.~Oliver, \emph{The homotopy theory of fusion systems},
  J. Amer. Math. Soc. \textbf{16} (2003), no.~4, 779--856.

\bibitem[CH21]{Chermak/Henke}
A.~Chermak and E.~Henke, \emph{Fusion systems and localities -- a dictionary},
  arXiv:1706.05343v3 (2021).

\bibitem[Che13]{Chermak:2013}
A.~Chermak, \emph{Fusion systems and localities}, Acta Math. \textbf{211}
  (2013), no.~1, 47--139.

\bibitem[Che17]{ChermakIII}
\bysame, \emph{{Finite Localities III}}, arXiv:1505.06161v2 (2017).

\bibitem[Che21a]{Chermak:2015}
\bysame, \emph{{Finite Localities I}}, arXiv:1505.07786v4 (2021).

\bibitem[Che21b]{ChermakII}
\bysame, \emph{{Finite Localities II}}, arXiv:1505.08110v3 (2021).

\bibitem[GL16]{Glauberman/Lynd}
George Glauberman and Justin Lynd, \emph{{Control of fixed points and existence
  and uniqueness of centric linking systems}}, Invent. Math. (2016), 1--44,
  Published online first: March 19 2016, doi:10.1007/s00222-016-0657-5.

\bibitem[Hen13]{Henke:2013}
E.~Henke, \emph{Products in fusion systems}, J. Algebra \textbf{376} (2013),
  300--319.

\bibitem[Hen15]{Henke:2015a}
\bysame, \emph{Products of partial normal subgroups}, Pacific J. Math.
  \textbf{279} (2015), no.~1-2, 255--268.

\bibitem[Hen19]{Henke:2015}
\bysame, \emph{Subcentric linking systems}, Trans. Amer. Math. Soc.
  \textbf{371} (2019), no.~5, 3325--3373.

\bibitem[Hen20]{Henke:2020}
\bysame, \emph{Extensions of homomorphisms between localities},
  arXiv:2006.12626 (2020).

\bibitem[Hen21]{Henke:Regular}
\bysame, \emph{Commuting partial normal subgroups and regular localities},
  arXiv:2103.00955v3 (2021).

\bibitem[Oli10]{O4}
B.~Oliver, \emph{Extensions of linking systems and fusion systems}, Trans.
  Amer. Math. Soc. \textbf{362} (2010), no.~10, 5483--5500.

\bibitem[Oli13]{Oliver:2013}
\bysame, \emph{Existence and uniqueness of linking systems: {C}hermak's proof
  via obstruction theory}, Acta Math. \textbf{211} (2013), no.~1, 141--175.

\bibitem[OV07]{OV1}
B.~Oliver and J.~Ventura, \emph{Extensions of linking systems with {$p$}-group
  kernel}, Math. Ann. \textbf{338} (2007), no.~4, 983--1043.

\end{thebibliography}

\end{document}